\documentclass[12pt,a4paper,oneside]{amsart} 

\usepackage[english]{babel}
\usepackage[utf8]{inputenc}
\usepackage{setspace}
\usepackage{enumitem}
\usepackage{microtype}
\usepackage[usenames,dvipsnames,svgnames,table]{xcolor}
\usepackage[textsize=footnotesize]{todonotes}

\usepackage{tikz}
\usepackage{tikz-cd}
\usetikzlibrary{arrows}
\usetikzlibrary{matrix}
\usetikzlibrary{calc}
\usepackage{tikz-qtree}
\usepackage{tikz-qtree-compat}

\usepackage{amsmath}
\usepackage{amsfonts}
\usepackage{amssymb}
\usepackage{amsthm}
\usepackage{amstext}
\usepackage{braket}
\usepackage{mathtools}
\usepackage{array}
\usepackage{enumitem}
\usepackage{mathrsfs}

\theoremstyle{plain}
\newtheorem{theorem}{\textbf{Theorem}}[section]
\newtheorem{proposition}[theorem]{\textbf{Proposition}}
\newtheorem{corollary}[theorem]{\textbf{Corollary}}	
\newtheorem{lemma}[theorem]{\textbf{Lemma}}

\theoremstyle{definition}

\newtheorem{remark}[theorem]{\textbf{Remark}}

\theoremstyle{plain}

\theoremstyle{plain}

\newcommand{\thistheoremname}{}
\newtheorem*{genericthm*}{\thistheoremname}
\newenvironment{namedthm*}[1]
{\renewcommand{\thistheoremname}{#1}%
	\begin{genericthm*}}
	{\end{genericthm*}}
\newenvironment{namedtheorem*}[1]
{\renewcommand{\thistheoremname}{#1}%
	\begin{genericthm*}}
	{\end{genericthm*}}

\numberwithin{equation}{section}

\newcommand{\df}{\overset{\text{def}}{=}}

\newcommand{\ZZ}{\mathbb{Z}}

\newcommand{\CC}{\mathbb{C}}

\newcommand{\OO}{\mathcal{O}}

\renewcommand{\a}{\alpha}

\newcommand{\bigosum}{\bigoplus}

\newcommand{\Pic}{\operatorname{\Pic}}
\newcommand{\Res}{\operatorname{\Res}}

\newcommand{\lra}{\longrightarrow}

\newcommand*{\lhra}{\ensuremath{\lhook\joinrel\relbar\joinrel\rightarrow}}

\usepackage[a4paper,top=3cm,bottom=3cm,left=0.5cm,right=0.5cm,bindingoffset=0mm]{geometry}
\usepackage{hyperref}
\mathtoolsset{showonlyrefs,showmanualtags}

\usepackage[style=math-alphabetic,maxbibnames=99]{biblatex}
\bibliography{./bibliografia}

\title{On the Prym map for cyclic covers of genus two curves}
\author{Daniele Agostini}
\address{H\"umboldt-Universit\"at zu Berlin, Institut f\"ur Mathematik, Unter den Linden 6, 10099 Berlin}
\curraddr{Universit\"at T\"ubingen, Geschwister-Scholl-Platz, 72072 T\"ubingen}
\email{daniele.agostini@uni-tuebingen.de}

\begin{document}

\begin{abstract}
	The Prym map assigns to each covering of curves a polarized abelian variety. In the case of {\'e}tale cyclic covers of curves of genus two, we show that the Prym map is ramified precisely on the locus of bielliptic covers. The key observation is that we can naturally associate to such a cover an abelian surface with a cyclic polarization, and then the codifferential of the Prym map can be interpreted in terms of multiplication of sections on the abelian surface. Furthermore, we give a different proof of a result by Ramanan that a genus two cyclic cover of degree sufficiently high is never hyperelliptic.
\end{abstract}

\maketitle

\section{Introduction}\label{sectionprym}

In this short note we study Prym varieties associated to cyclic covers of genus two curves. Recall that, if $f\colon C \lra D$ is finite cover of smooth projective curves, we define the associated \textit{Prym variety} as the principal connected component of the kernel of the induced norm map:
\begin{equation}
\operatorname{Prym}(C\to D): \df \operatorname{Ker}^0\left[ \operatorname{Nm}(f)\colon \operatorname{Pic}^0(C) \lra \operatorname{Pic}^0(D) \right].
\end{equation}
Prym varieties are a classic subject of algebraic geometry, and they have been intensely studied, especially in the case of {\'e}tale double covers: we refer to  \cite[Chapter 12]{BirkenhakeLange2004}, \cite{BeauvillePrymSurvey1989} and \cite{FarkasPrym2012} for an overview of the topic. 

In recent years, there has been an intense work on Prym varieties \cite{MarcucciPirola2012,MarcucciNaranjo2014,ColomboetalShimura2019,NaranjoOrtegaVerra2019}. In particular, Lange and Ortega \cite{Ortega2003,LangeOrtegaPrymCyclic2010,LangeOrtegaPrym72016},\cite{LangeOrtegaPrymHyp2018} have studied Prym varieties associated to cyclic {\'e}tale covers of genus two curves. More precisely, let $C \to D$ be a cyclic {\'e}tale cover of degree $d$ of a genus two curve. Then the corresponding Prym variety has a natural polarization, obtained as the restriction of the natural principal polarization on the Jacobian of $C$. It turns out that the type $\delta$ of the polarization depends only on the degree $d$ of the cover, hence we get a \textit{Prym map}
\begin{equation}\label{eq:prymmap}
\operatorname{Pr}\colon \mathcal{R}_{2,d} \lra \mathcal{A}_{\delta}, \qquad [C\to D] \mapsto [\operatorname{Prym}(C\to D)]
\end{equation}
from the moduli space of {\'e}tale cyclic covers of degree $d$ of a genus two curve, to the moduli space of abelian varieties with a polarization of type $\delta$. 

In particular, Lange an Ortega proved in \cite{LangeOrtegaPrymCyclic2010} that the differential of the Prym map for $d=7$ is injective at a general point, so that the map is generically finite onto its image. Here, we use syzygies of abelian surfaces to extend this result to $d\geq 6$ and we moreover characterize the covers where the differential of the Prym map is not injective.

\begin{namedtheorem*}{Theorem A}
	The differential of the Prym map $\operatorname{Pr}\colon \mathcal{R}_{2,d} \lra \mathcal{A}_{\delta}$ is injective at a cyclic cover in $\mathcal{R}_{2,d}$ if and only if $d\geq 6$ and the cover is not bielliptic. In particular, the Prym map is generically finite onto its image for $d\geq 6$.	
\end{namedtheorem*}  

Recall that a cover $f\colon C \to D$ is said to be \emph{bielliptic} if  there exist compatible bielliptic quotients of $C$ and $D$. This means that there is a commutative diagram
	\begin{equation}
	\begin{tikzcd}
	C \arrow{r}{2:1}\arrow{d}{f} & E \arrow{d}{} \\
	D \arrow{r}{2:1} & F
	\end{tikzcd}
	\end{equation}
where $E$ and $F$ are elliptic curves and $C\to E$ and $D\to F$ are double covers.	In particular, a straightforward parameter count shows that the locus of bielliptic covers is a divisor in the moduli space $\mathcal{R}_{2,d}$ and then the theorem shows that the Prym map is generically finite onto its image whenever $d\geq 6$. 

In the course of the proof of Theorem A, we will also show that the covering curve is never hyperelliptic, whenever $d\geq 7$. This result was already observed  by Ramanan in \cite[Section 3]{RamananAmpleAbSurf1985} for cyclic covers of degrees $d\geq 3$, but we give here a proof along different lines. I thank Pawe{\l} Bor\'owka for pointing me to Ramanan's result and also for pointing out the similar result in \cite[Proposition 2.3]{Borowka} for cyclic covers of curves of higher degree.

\begin{namedtheorem*}{Theorem B}
	Let $C\to D$ be a cyclic cover of a genus two curve of degree $d\geq 7$. Then the curve $C$ is not hyperelliptic.	
\end{namedtheorem*}

To prove both results, we first describe a construction that associates to a cyclic cover $[C\to D] \in \mathcal{R}_{2,d}$ a polarized abelian surface $(A,L)$ of type $(1,d)$. Then, we show that the codifferential of the Prym map at $[C\to D]$ can be interpreted via the multiplication map $\operatorname{Sym}^2H^0(A,L) \to H^0(A,L^{2})$. We then use results of Lazarsfeld \cite{LazarsfeldProjective1990} and Fuentes Garc\'ia \cite{GarciaProjective2004}, which show that this multiplication map is surjective if and only if $d\geq 7$ and the corresponding polarization is very ample, and then a result of Ramanan \cite{RamananAmpleAbSurf1985} proves that this is the case if and only if the cover is not bielliptic. A similar argument yields the case $d=6$ of Theorem A, and to conclude we analyze the case of bielliptic covers.

\vspace{10pt}
\textbf{Acknowledgements}: I would like to thank the anonymous referee, whose comment led us to a correction of the original argument and to a strenghtening of our results.
The main ideas in this note were developed during my PhD at the Humboldt-Universit\"at zu Berlin. I would like to thank my advisor Prof.  Gavril Farkas for his guidance and the Berlin Mathematical School and the IRTG 1800 of the DFG for their  support during my studies. I would also like to thank Angela Ortega for her useful comments on this manuscript. I thank Pawe{\l} Bor\'owka for pointing out Ramanan's result to me.

\section{Background}

\subsection{General facts on cyclic covers}
We collect here some facts about {\'e}tale cyclic covers of smooth varieties. Let $X$ be a smooth quasiprojective variety, together with a free action of the cyclic group $\ZZ/d\ZZ$. Since the action is free, the quotient $Y$ is again a smooth and irreducible projective variety, and the quotient map $f\colon X\lra Y$ is finite and {\'e}tale of degree $d$ \cite[Theorem p.66]{MumfordAbelianBook1974}. Such a map $f\colon X \lra Y$ is called a \textit{cyclic {\'e}tale cover of degree $d$}. 

Furthermore the pushforward $f_*\OO_X$ has a structure of $\OO_Y$-algebra that decomposes according to the irreducible representations of $\ZZ/d\ZZ$: this decomposition has the form  
\begin{equation}
f_*\OO_X \cong \bigoplus_{i=0}^{d-1} \eta^{-i}
\end{equation} 
where $\eta\in \operatorname{Pic}(Y)$ is a $d$-torsion line bundle, meaning that there is an isomorphism $\eta^{d} \to \OO_Y$.  We note that the $\OO_Y$-algebra structure on $f_*\OO_X$ gives us one specific isomorphism $\varphi\colon \eta^{d} \to \OO_Y$. 

Conversely, take a $d$-torsion line bundle $\eta$ on a smooth quasiprojective variety $Y$. If we fix an isomorphism $\varphi\colon \eta^d \to \OO_Y$, we can endow the sheaf $\bigoplus_{i=0}^{d-1}\eta^{-i}$ with the structure of a $\OO_Y$-algebra, and it is easy to see that
\begin{equation}\label{pushforwardcycliccovers} 
\operatorname{Spec} \bigoplus_{i=0}^{d-1} \eta^{-i} \lra Y
\end{equation} 
is a cyclic {\'e}tale cover of degree $d$. Hence, there is a correspondence between cyclic {\'e}tale covers of degree $d$ over $Y$ and $d$-torsion line bundles $\eta$, together with an isomorphism $\eta^d \to \OO_Y$. Moreover, if we choose two different isomorphisms $\eta^d \to \OO_Y$, it is easy to see that the corresponding cyclic covers of $Y$ are isomorphic.

\begin{remark}\label{rmk:projectionmap}
Let now $f\colon X\to Y$ be a cyclic cover induced by a torsion line bundle $\eta$ of order $d$ on $Y$. If $\mathcal{L}$ is any line bundle on $Y$, the projection formula yields a decomposition $f^*f_*\mathcal{L} \cong \oplus_{i=0}^{d-1} \mathcal{L}\otimes \eta^{-i}$, which corresponds to the decomposition of $f_*f^*\mathcal{L}$ according to the action of $\mathbb{Z}/d\mathbb{Z}$. In particular, taking $(\mathbb{Z}/d\mathbb{Z})$-invariants we have a projection $(-)^{(\mathbb{Z}/d\mathbb{Z})}\colon f_*f^*\mathcal{L}\to \mathcal{L}$, of which the usual inclusion map $\mathcal{L} \hookrightarrow f_*f^*\mathcal{L}$ is a section. We observe that the same decomposition carries over to the global sections. 
\end{remark} 

We will need later the following well-known lemma, of which we give a proof for completeness.

\begin{lemma}\label{kernelpullback} 
Suppose that $Y$ is projective and connected and let $f\colon X \to Y$ be an {\'e}tale cyclic cover given by a $d$-torsion line bundle $\eta$. Then, the kernel of the pullback map $f^*\colon \operatorname{Pic}(Y) \to \operatorname{Pic}(X)$ is precisely the subgroup generated by $\eta$.
\end{lemma}
\begin{proof}
Suppose that $\mathcal{L}$ is a line bundle on $Y$ such that	$f^*\mathcal{L} \cong \OO_X$. The projection formula gives that $\mathcal{L}\otimes f_*\OO_X \cong f_*\OO_X$, so that
\begin{equation}
\bigoplus_{i=0}^{d-1}\mathcal{L}\otimes \eta^{-i} \cong \bigoplus_{i=0}^{d-1} \eta^{-i}.
\end{equation}
Then, by the Krull-Schmidt theorem \cite[Theorem 1,Theorem 3]{AtiyahKrullSchmidt1956}, it follows that $\mathcal{L} \cong \eta^{-i}$, for a certain $i$. For the converse, we need to prove that $f^*\eta \cong \OO_X$. However, the previous reasoning shows that $h^0(X,f^*\eta)\ne 0$ and $h^0(X,f^*\eta^{-1}) \ne 0$: this way we get two injective maps $\OO_X\to f^*\eta \to \OO_X$, and since the composition is an isomorphism, it follows that both of them are isomorphisms as well.
\end{proof}

\subsection{Cyclic covers of curves and the Prym map}
We can specialize the previous discussion to smooth curves: let $D$ be a smooth and irreducible curve of genus $g$. Then, by what we have remarked before, isomorphism classes of {\'e}tale cyclic covers of degree $d$ of $D$ correspond to $d$-torsion line bundles $\eta \in \operatorname{Pic}^0(D)$. 
Moreover, it is easy to see that a cover $f\colon C\to D$ corresponding to $\eta$ is connected if and only if $\eta$ has order precisely $d$.

From this discussion, we have a \textit{moduli space of {\'e}tale cyclic covers of degree $d$} as the space $\mathcal{R}_{g,d}$ of isomorphism classes $[D,\eta]$, where $D$ is a smooth curve of genus $g$ and $\eta \in \operatorname{Pic}^0(D)$ is a torsion bundle of order $d$. Such a couple $(D,\eta)$ is sometimes called also a \textit{level curve} of order $d$.

\begin{remark}
	The space $\mathcal{R}_{g,d}$ is irreducible \cite{BernsteinThesis1999} and since each curve has a finite number of torsion line bundles, we see that $\dim \mathcal{R}_{g,d} = \dim \mathcal{M}_{g} = 3g-3$. 
\end{remark}

At this point, consider a level curve $[D,\eta]\in \mathcal{R}_{g,d}$ and let $f\colon C\to D$ be its corresponding cyclic cover. Then we can take the Prym variety of this cover, that we denote by $\operatorname{Prym}(D,\eta)$. This is a polarized abelian variety of type \cite[Corollary 12.1.5, Lemma 12.3.1]{BirkenhakeLange2004}
\begin{equation}
\label{type}
	\delta = (1,1,1,\dots,1,d,d,d,\dots,d)
\end{equation}
where $1$ is repeated $(d-2)(g-1)$ times and $d$ is repeated $g-1$ times.

Let us denote by $\mathcal{A}_{\delta}$ the moduli space of abelian varieties with a polarization of the type \eqref{type}. Then the Prym construction gives a map of moduli spaces, called the \emph{Prym map}:

	\begin{equation}
	\operatorname{Pr}_{g,d} \colon \mathcal{R}_{g,d} \lra \mathcal{A}_{\delta}, \qquad [D,\eta] \mapsto [\operatorname{Prym(D,\eta)}].
	\end{equation}

Lange and Ortega have considered in \cite{LangeOrtegaPrymCyclic2010,LangeOrtegaPrym72016} the differential of the Prym map for cyclic covers and they have proved that it is very often injective. In particular, it follows that the Prym map is generically finite. 

Here we want to describe this differential, following \cite{LangeOrtegaPrymCyclic2010}.  Consider again a level curve $[D,\eta] \in \mathcal{R}_{g,d}$ and let $f\colon C \to D$ be a corresponding cyclic cover. Since the cover is {\'e}tale, we have $f^*\omega_D \cong \omega_C$, and the projection formula together with \eqref{pushforwardcycliccovers} gives
\begin{equation}
f_*\omega_C \cong \omega_D \otimes f_*\OO_C \cong \bigoplus_{i=0}^{d-1} \omega_D\otimes \eta^{-i}.
\end{equation}
Taking global sections, we get that
\begin{equation}\label{decompositionomegaC}
H^0(C,\omega_C) = \bigoplus_{i=0}^{d-1} H^0(D,\omega_D\otimes \eta^{-i})
\end{equation}
and this is exactly the decomposition of $H^0(C,\omega_C)$ into $(\ZZ/d\ZZ)$-representations. We single out the non-trivial representations and we set
\begin{equation}\label{definitionW} 
W:\df \bigoplus_{i=1}^{d-1} H^0(D,\omega_D\otimes \eta^{-i}) \subseteq H^0(C,\omega_C).
\end{equation}
With this, we can state the result about the differential of the Prym map.

\begin{proposition}[Lange-Ortega]\label{differentialprym} 
	With the above notation, the dual of the differential of the Prym map at $[D,\eta] \in\mathcal{R}_{g,d}$ is the composition
	\begin{equation}
	\operatorname{Sym}^2 W \overset{m}{\lra} H^0(C,\omega_C^{2}) \overset{(-)^{\mathbb{Z}/d\mathbb{Z}}}{\lra} H^0(D,\omega_D^2),
	\end{equation} 
	where the first map is the multiplication of sections, and the second map is the projection to the $(\mathbb{Z}/d\mathbb{Z})$-invariant part.	
\end{proposition}
\begin{proof}
	See \cite[Proposition 4.1]{LangeOrtegaPrymCyclic2010}.	
\end{proof}

\begin{remark}\label{rmk:invariantW}
	Since the multiplication map is $(\mathbb{Z}/d\mathbb{Z})$ equivariant, in Proposition \ref{differentialprym}, we only need to take into consideration the $(\mathbb{Z}/d\mathbb{Z})$-invariant part of $\operatorname{Sym}^2 W$. Thus, the differential of the Prym map at $[D,\eta]$ is injective if and only if the multiplication map
	\[ \left( \operatorname{Sym}^2 W \right)^{\mathbb{Z}/d\mathbb{Z}} = \bigoplus_{\substack{1\leq i \leq j \leq d-1 \\ i+j=d }} H^0(D,\omega_D\otimes \eta^i)\otimes H^0(D,\omega_D\otimes \eta^j) \lra H^0(D,\omega_D^2) \]
	is surjective.
\end{remark}

\section{The non bielliptic case}

In this section we show how to associate a polarized abelian surface to a cyclic cover of a genus two curve. We then use this abelian surface to study the differential of the Prym map and to prove Theorems A and B whenever the cover is not bielliptic.  The construction of the abelian surface is quite standard, and can be found also in \cite{BirkenhakeLange2004,RamananAmpleAbSurf1985}, but we would like to outline the necessary steps nevertheless.

Take a level curve $[D,\eta] \in\mathcal{R}_{2,d}$ and let $JD = \operatorname{Pic}^0(D)$ be the Jacobian variety of $D$. We fix a point $P_0\in D$, so that we have the corresponding Abel-Jacobi map
\begin{equation}
\a \colon D \lhra JD, \qquad P\mapsto \mathcal{O}_D(P-P_0)
\end{equation}
which realizes $D$ as a divisor on $JD$. Standard properties of the Abel-Jacobi map imply that the line bundle $M = \OO_{JD}(D)$ is a principal polarization on $JD$ \cite[Corollary 11.2.3]{BirkenhakeLange2004} and also that the pullback map
\begin{equation}
\a^* \colon \text{Pic}^0(JD) \lra \text{Pic}^0(D)
\end{equation}
is an isomorphism \cite[Lemma 11.3.1]{BirkenhakeLange2004}.  
In particular, the line bundle $\eta_{JD} :\df (\a^*)^{-1}(\eta)$ on $JD$ is again a torsion bundle of order $d$. If we choose an isomorphism $\varphi_{JD}\colon \eta_{JD}^d \to \OO_{JD}$, we can pull it back via $\alpha$ to an isomorphism $\varphi:\df \alpha^*(\varphi_{JD})\colon \eta^d \to \OO_D$. We can take the corresponding cyclic covers,
\begin{equation}
C :\df \text{Spec} \bigosum_{i=0}^{d-1}\eta^{-i}, \qquad \qquad A :\df \text{Spec} \bigosum_{i=0}^{d-1} \eta_{JD}^{-i}
\end{equation}
and we have the following:

\begin{lemma}\label{coversabelian}
With the above notation, we have a fibered square	
\begin{center}
	\begin{tikzpicture}[node distance=1.5cm, auto]
	\node(A) {$C$};
	\node(B) [right of= A] {$A$};
	\node(A') [below of = A]{$D$};
	\node(B')[right of= A']{$JD$};
	\draw[->] (A) to node {$j$} (B);
	\draw[->] (A') to node {$\a$} (B');
	\draw[->] (A) to node [left] {$f$} (A');
	\draw[->] (B) to node  {$h$} (B');
	\end{tikzpicture}
\end{center}
Moreover, $A$ is an abelian surface, the line bundle $L:\df h^*M$ is ample of type $(1,d)$ and under the embedding $j\colon C\hookrightarrow A$ the curve $C$ can be considered as a divisor $C\in |L|$.
\end{lemma}
\begin{proof}
By construction we have an isomorphism $\alpha^*\left( \bigoplus_{i=0}^{d-1}\eta_{JD}^{-i} \right) \cong \bigoplus_{i=0}^{d-1}\eta^{-i}$ of sheaves of $\OO_D$-algebras. Hence, we get a fibered square as above from the properties of the relative Spec. To see that $A$ is an abelian surface, one observes first that it is connected, since $\eta_{JD}$ has exactly order $d$, and then $h^0(A,\OO_A) = \sum_{i=0}^{d-1}h^0(JD,\eta_{JD}^{-i}) = 1$.  Since $h\colon A \lra JD$ is an {\'e}tale finite map, and $A$ is connected, it follows from the Serre-Lang theorem \cite[Theorem IV.18]{MumfordAbelianBook1974} that $A$ is an abelian surface and that the map $F$ is an isogeny.  

To conclude, we need to prove that $L$ is ample and of type $(1,d)$. Since the map $h$ is finite and $M$ is ample, it follows from \cite[Proposition 1.2.13]{LazarsfeldPositivityI2004} that $L=h^*M$ is ample. For the type, it is enough to prove that $\operatorname{Ker} h\cong \ZZ/d\ZZ$ \cite[Lemma 3.1.5]{BirkenhakeLange2004}. Thanks to \cite[Proposition 2.4.3]{BirkenhakeLange2004}, it is enough to show the same for $\operatorname{Ker} h^*$: however we know from Lemma \ref{kernelpullback}, that $\operatorname{Ker} h^*$ is precisely the subgroup generated by $\eta$, so that  $\operatorname{Ker} h^* \cong \ZZ/d\ZZ$.
\end{proof}

Recall from \eqref{decompositionomegaC} that we have a decomposition $H^0(C,\omega_C) = \bigoplus_{i=0}^{d-1} H^0(D,\omega_D\otimes \eta^{-i})$ into $(\ZZ/d\ZZ)$-representations. We have defined in \eqref{definitionW} the linear system $W=\bigoplus_{i=1}^{d-1} H^0(D,\omega_D\otimes \eta^{-i})$ and now we want to give an interpretation of it in terms of the abelian surface $A$.

\begin{lemma}\label{Wasrestriction}
With notations as before, $W$ coincides with the image of the restriction map from $H^0(A,L)$ to $H^0(C,\omega_C)$:
\begin{equation}
W = \operatorname{Im}\left( H^0(A,L) \lra H^0(C,\omega_C) \right).
\end{equation}	
\end{lemma}
\begin{proof}
First we observe that the restriction map makes sense, since $C\in |L|$ by Lemma \ref{coversabelian}, so that the adjunction formula gives $\omega_C \cong \omega_A\otimes L_{|C} \cong L_{|C}$. Let $\tau \in H^0(JD,M)$ be a section such that $D = \{ \tau = 0 \}$. Using again the adjunction formula, we have an exact sequence of sheaves on $JD$
\begin{equation}
0 \lra \OO_{JD} \overset{\cdot \tau}{\lra} M \lra \omega_D \lra 0
\end{equation}
and Lemma \ref{coversabelian} shows that, pulling back via $h^*$, we get an exact sequence
\begin{equation} 
0 \lra \OO_A \overset{\cdot \sigma}{\lra} L \lra \omega_C \lra 0
\end{equation}
where $\sigma := h^*(\tau) \in H^0(A,L)$. By construction, we see that this is actually an exact sequence of sheaves, together with a $(\ZZ/d\ZZ)$-action and, moreover, if we take the pushforward along $h_*$ and then take the $(\ZZ/d\ZZ)$-invariant part as in Remark \ref{rmk:projectionmap}, we get a commutative diagram  with exact rows:
\begin{equation}
\begin{tikzcd}
0 \arrow{r}{} & h_*(\OO_A) \arrow{r}{h_*(\cdot \sigma)}\arrow{d}{(-)^{\ZZ/d\ZZ}} & h_*L \arrow{r}{}\arrow{d}{(-)^{\ZZ/d\ZZ}} & h_*\omega_C \arrow{r}{}\arrow{d}{(-)^{\ZZ/d\ZZ}} & 0 \\
0 \arrow{r}{} & \OO_{JD} \arrow{r}{\cdot \tau} & M \arrow{r}{} & \omega_D \arrow{r}{} & 0.
\end{tikzcd}
\end{equation}
Passing to global sections, we get another commutative diagram with exact rows:
\begin{equation}\label{diagramglobalsections} 
\begin{tikzcd}
0 \arrow{r}{} & \CC\sigma \arrow{r}{}\arrow{d}{(-)^{\ZZ/d\ZZ}} & H^0(A,L) \arrow{r}{}\arrow{d}{(-)^{\ZZ/d\ZZ}} & H^0(C,\omega_C) \arrow{r}{}\arrow{d}{(-)^{\ZZ/d\ZZ}} & H^1(A,\OO_A) \arrow{r}{}\arrow{d}{(-)^{\ZZ/d\ZZ}} & 0 \\
0 \arrow{r}{} & \CC\tau \arrow{r}{} & H^0(JD,M) \arrow{r}{} & H^0(D,\omega_D) \arrow{r}{} & H^1(JD,\OO_{JD}) \arrow{r}{} &  0
\end{tikzcd}
\end{equation}
Since $M$ gives a principal polarization, we have that $h^0(JD,M)=1$, so that the map $\CC\tau \to H^0(JD,M)$ is an isomorphism. Hence, the map $H^0(JD,M)\to H^0(D,\omega_D)$ is zero, and since the diagram is commutative, it follows that  
\begin{equation}\label{inclusion}
\operatorname{Im}\left( H^0(A,L) \lra H^0(C,\omega_C) \right) \subseteq W = \operatorname{Ker}\left( (-)^{\ZZ/d\ZZ}\colon H^0(C,\omega_C) \lra H^0(D,\omega_D)  \right).
\end{equation}
To conclude it is enough to show that the two spaces in \eqref{inclusion} have the same dimension. To prove this, we look again at diagram \eqref{diagramglobalsections} and we see that
\begin{small} 
\begin{align}
\dim_{\CC} \operatorname{Im}\left( H^0(A,L) \lra H^0(C,\omega_C) \right) & = h^0(C,\omega_C) - h^1(A,\OO_A) = h^0(C,\omega_C) - 2 ,\\	
\dim_{\CC} W &  = h^0(C,\omega_C) - h^0(D,\omega_D) = h^0(C,\omega_C) - 2.
\end{align}
\end{small}
\end{proof}

With this lemma, we can reinterpret the codifferential of the Prym map in Proposition \ref{differentialprym} as a multiplication map on the abelian surface $A$:

\begin{lemma}\label{differentialprym2} 
With the same notations of before, we have that 
\begin{equation}
\operatorname{Sym}^2 W \lra H^0(C,\omega_C^2) \text{ is surjective }  
\end{equation}	
if and only if
\begin{equation}
\operatorname{Sym}^2 H^0(A,L) \lra H^0(A,L^2) \text{ is surjective. }
\end{equation} 
More precisely, the cokernels of the two maps have the same dimension.
\end{lemma}
\begin{proof}
 We first observe that in the statement we can replace $\operatorname{Sym}^2 W$ and $\operatorname{Sym}^2 H^0(A,L)$ with $W^{\otimes 2}$ and $H^0(A,L)^{\otimes 2}$ respectively. We take again a section $\sigma \in H^0(A,L)$ such that $C=\{ \sigma = 0 \}$: then Lemma \ref{Wasrestriction} gives the exact sequence
\begin{equation}\label{exactsequenceW} 
0 \lra \CC\sigma \lra H^0(A,L) \lra W \lra 0.
\end{equation}
Instead, if we take global sections in the exact sequence
\begin{equation}
0 \lra L \overset{\cdot \sigma}{\lra} L^2 \lra \omega_C^2 \lra 0
\end{equation}
and we use that $H^1(A,L)=0$ by Kodaira vanishing, we get an exact sequence
\begin{equation}\label{exactsequenceL} 
0 \lra H^0(A,L)\sigma \lra H^0(A,L^2) \lra H^0(C,\omega_C^2) \lra 0.
\end{equation}
Putting together \eqref{exactsequenceW} and \eqref{exactsequenceL}, we get a commutative diagram with exact rows
\begin{equation}
\begin{tikzcd}
0 \arrow{r}{} & \sigma\otimes H^0(A,L) + H^0(A,L)\otimes \sigma \arrow{r}{}\arrow{d}{} & H(A,L)^{\otimes 2} \arrow{r}{}\arrow{d}{} & W^{\otimes 2} \arrow{r}{}\arrow{d}{} & 0 \\
0 \arrow{r}{} & H^0(A,L)\sigma \arrow{r}{} & H^0(A,L^2) \arrow{r}{} & H^0(C,\omega_C^2) \arrow{r}{} & 0
\end{tikzcd}
\end{equation}
Since the map
\begin{equation}
\sigma\otimes H^0(A,L) + H^0(A,L)\otimes \sigma \lra H^0(A,L)\sigma 
\end{equation}
is clearly surjective, the Snake Lemma proves that
\begin{equation}
\operatorname{Coker}\left( H^0(A,L)^{\otimes 2} \lra H^0(A,L^2) \right) \cong \operatorname{Coker}\left( W^{\otimes 2} \lra H^0(C,\omega_C^2) \right)
\end{equation}
which implies the statement. 	
\end{proof}

Now, as outlined in the introduction, we can prove Theorems A and B, when the cover is not bielliptic.

\begin{corollary}\label{cor:notbielliptic}
	Suppose that $d\geq 6$ and that the cover $C\lra D$ is not bielliptic. Then the differential of the Prym map \eqref{eq:prymmap} is injective at $[D,\eta]$. Moreover, if $d\geq 7$, the curve $C$ is not hyperelliptic.
\end{corollary}
\begin{proof}
Let's suppose first that $d\geq 7$. Then for both statements, it is enough to prove that the multiplication map 
\begin{equation}\label{eq:mapm} 
m\colon \operatorname{Sym}^2 W \lra H^0(C,\omega_C^2)
\end{equation}  
is surjective. Indeed, since the projection map $(-)^{\mathbb{Z}/d\mathbb{Z}}\colon H^0(C,\omega_C^2) {\to} H^0(D,\omega_D^2)$ is always surjective, we would get that the composition in Proposition \ref{differentialprym} is surjective as well, meaning that the differential of the Prym map is injective at $[D,\eta]$. Furthermore, if \eqref{eq:mapm} is surjective, the multiplication map $\operatorname{Sym}^2 H^0(C,\omega_C) \to H^0(C,\omega_C^2)$ must be a fortiori surjective, and since $C$ has genus $d+1\geq 8$, Noether's theorem shows that the curve $C$ cannot be hyperelliptic.
 
To prove that \eqref{eq:mapm} is surjective, we can use Lemma \ref{differentialprym2} and show that $\operatorname{Sym}^2 H^0(A,L) \to H^0(A,L^2)$ is surjective, where $(A,L)$ is the polarized abelian surface which corresponds to $(C,\eta)$, as in Lemma \ref{coversabelian}. By results of \cite{LazarsfeldProjective1990} and Fuentes Garc\'ia \cite{GarciaProjective2004}, this happens if and only if $d\geq 7$ and the line bundle $L$ is very ample. To conclude, a theorem of Ramanan  \cite[Theorem 3.1]{RamananAmpleAbSurf1985} shows that $L$ fails to be very ample if and only if the original cover $C\to D$ is bielliptic.

Now let's suppose that $d=6$ and that the cover $C\to D$ is not bielliptic. According to Remark \ref{rmk:invariantW}, we need to prove that the multiplication map $(\operatorname{Sym}^2 W)^{(\mathbb{Z}/6\mathbb{Z})} \to H^0(D,\omega_D^2)$ is surjective. However, using the decomposition in Remark \ref{rmk:invariantW}, it is easy to see that $\dim  (\operatorname{Sym}^2 W)^{(\mathbb{Z}/6\mathbb{Z})} = 3 = \dim H^0(D,\omega_D^2)$, so that the map is surjective if and only if it is injective.	In particular, it suffices to prove that the map $\operatorname{Sym}^2 W \to H^0(C,\omega_C^2)$ is injective. We can interpret this as a statement on the dimension of the cokernel and then using Lemma \ref{differentialprym2} and some short computations, we see that this is equivalent to the injectivity of $\operatorname{Sym}^2 H^0(A,L) \to H^0(A,L^2)$ on the abelian surface $A$. However, since the cover is not bielliptic, Ramanan's theorem \cite[Theorem 3.1]{RamananAmpleAbSurf1985} proves that the line bundle $L$ is very ample, and then a result of Fuentes Garc\'ia \cite[Theorem 5.3]{GarciaProjectiveArxiv2003} shows that the map $\operatorname{Sym}^2 H^0(A,L) \to H^0(A,L^2)$ is injective.
\end{proof}

\section{The bielliptic case}

To conclude, we need to study what happens in the bielliptic case. Recall that in this case the cyclic cover $f\colon C\to D$ fits into a fibered square
\begin{equation}\label{eq:biellipticsquare}
\begin{tikzcd}
C \arrow{r}{}\arrow{d}{f} & E \arrow{d}{} \\
D \arrow{r}{g} & F
\end{tikzcd}
\end{equation}
where $E$ and $F$ are elliptic curves, the horizontal maps are double covers and the vertical maps are {\'e}tale cyclic covers of degree $d$. Equivalently, this means that the  line bundle $\eta$ on $D$ is the pullback $\eta=g^*\eta_F$ of a line bundle of order exactly $d$ on $F$. 

In particular, the covering curve $C$ is itself bielliptic, and this concludes the proof of Theorem B.

\begin{proof}[Proof of Theorem B]
	Let $C\to D$ be an {\'e}tale cyclic cover of degree $d\geq 7$, where $D$ is a smooth curve of genus two. Then Corollary \ref{cor:notbielliptic} shows that $C$ is not hyperelliptic whenever the cover is not bielliptic. If instead the cover is bielliptic, then the curve $C$ itself is bielliptic of genus $d+1\geq 8$ and it is a consequence of the Castelnuovo-Severi inequality \cite[Corollary p. 26]{KaniCastelnuovo1984}, that no bielliptic curve of genus at least $4$ can be hyperelliptic.
\end{proof}

In the rest of the section, we conclude the proof of Theorem A. More precisely, we need to prove that for a bielliptic cover, the differential of the Prym map is not injective. Equivalently, via Remark \ref{rmk:invariantW} we need to prove that the multiplication map
\begin{equation}\label{eq:multmapW}
\bigoplus_{\substack{1\leq i \leq j \leq d-1 \\ i+j=d }} H^0(D,\omega_D\otimes \eta^i)\otimes H^0(D,\omega_D\otimes \eta^j) \lra H^0(D,\omega_D^2)
\end{equation}
is not surjective. We first set up some notation: since $D$ has genus two, the double cover $g\colon D\to F$ is branched over two points, and standard facts about double covers show that $f_*\mathcal{O}_D \cong \mathcal{O}_C \oplus \mathcal{O}_C(-\delta)$, for a certain divisor $\delta$ on $F$ of degree $1$. In particular, if $\mathcal{M}$ is any line bundle on $F$, then we have a decomposition $H^0(D,f^*\mathcal{M})=H^0(F,\mathcal{M})\oplus H^0(F,\mathcal{M}-\delta)$, which corresponds to the decomposition into invariant and anti-invariant parts, under the bielliptic involution that induces the cover $D\to F$. 

We now set $\mathcal{L}=\mathcal{O}_F(\delta)$: other standard facts on double covers show that $\omega_D \cong g^*\mathcal{L}$ and in particular
\begin{equation}
H^0(D,\omega_D^2) = H^0(D,g^*(\mathcal{L}^2)) \cong H^0(F,\mathcal{L}^2) \oplus H^0(F,\mathcal{L}^2-\delta) 
\end{equation} 
Furthermore, since the cover is bielliptic, we know that $\eta \cong g^*\eta_F$ by construction.  Hence, for every $i=1,\dots,d-1$ we have
\begin{equation} 
H^0(D,\omega_D\otimes \eta^i) = H^0(D,g^*(\mathcal{L}\otimes \eta_F^i)) \cong H^0(F,\mathcal{L}\otimes \eta_F^i) \oplus H^0(F,\mathcal{L}\otimes \eta_F^i(-\delta)) = H^0(F,\mathcal{L}\otimes \eta_F^i) .
\end{equation}
where the last equality comes from the fact that $H^0(F,\mathcal{L}\otimes \eta_F^i(-\delta)) \cong H^0(F,\eta_F^i) = 0$. In conclusion, we can reinterpret the map \eqref{eq:multmapW} as a  map
\begin{equation}
\bigoplus_{\substack{1\leq i \leq j \leq d-1 \\ i+j=d }} H^0(F,\mathcal{L}\otimes \eta_F^i)\otimes H^0(F,\mathcal{L}\otimes \eta_F^j) \lra H^0(F,\mathcal{L}^2)\oplus H^0(F,\mathcal{L}^2-\delta)
\end{equation}
At this point, we claim that the image of this map must be contained in the subspace $H^0(F,\mathcal{L}^2)$. One way to see this is to interpret the map in terms of multiplication of sections on the curve $F$. Alternatively, we can observe that the domain is invariant with respect to the bielliptic involution on $D$, and that the invariant part on the codomain is precisely $H^0(F,\mathcal{L}^2)$. To conclude, we simply observe  that $H^0(F,\mathcal{L}^2-\delta) \cong H^0(F,\mathcal{L}) \cong \mathbb{C}$ so that the map cannot be surjective.

This allows us to conclude the proof of Theorem A.

\begin{proof}[Proof of Theorem A]
	Let $D$ be a curve of genus two and let $C\to D$ be an {\'e}tale cyclic cover of degree $d\geq 6$. If the cover is not bielliptic Corollary \ref{cor:notbielliptic} asserts that the Prym map is unramified at $[C\to D]$. If instead the cover is bielliptic, then the discussion in this last section shows that the differential of the Prym map is not injective at $[C\to D]$ and this concludes the proof.
\end{proof}

\printbibliography

\end{document}